\documentclass[11pt]{amsart}
\usepackage{latexsym,amssymb,amsmath,amsthm,comment,graphicx,color,morefloats,xypic,listings}
\usepackage[section]{placeins}
\newtheorem{thm}{Theorem}

\newtheorem{lemma}[thm]{Lemma}
\newtheorem{propo}[thm]{Proposition}
\theoremstyle{definition}

\def\B#1#2{{#1\choose #2}}
\def\Gcal{\mathcal{G}}
\newcommand{\D}{\mathbb{D}} \newcommand{\PP}{\mathbb{P}} 
 \newcommand{\R}{\mathbb{R}} 
\newcommand{\Z}{\mathbb{Z}} \newcommand{\T}{\mathbb{T}} \newcommand{\Q}{\mathbb{Q}}
\def\G{\mathcal{G}}   
\def\O{\mathcal{O}} \def\B{\mathcal{B}} 
 \def\dim{\rm{dim}} \def\E{\mathcal{E}}  
\title{Universality for Barycentric subdivision}
\author{Oliver Knill}
\date{September 20, 2015}
\address{ Department of Mathematics \\ Harvard University \\ Cambridge, MA, 02138 }

\begin{document}
\maketitle

\begin{abstract}
The spectrum of the Laplacian of successive Barycentric subdivisions of a graph 
converges exponentially fast to a limit which only depends 
on the clique number of the initial graph and not on the graph itself. 
Announced in \cite{KnillBarycentric}), the proof uses now an 
explicit linear operator mapping the clique vector of a graph 
to the clique vector of the Barycentric refinement. The eigenvectors of its
transpose produce integral geometric invariants for which Euler characteristic is 
one example.
\end{abstract}

\section{Notations}

Given a finite simple graph $G=(V,E)$ with vertex set $V$ and edge set $E$,
the {\bf Barycentric refinement} $G_1=(V_1,E_1)$ is the graph for which $V_1$ 
consists of all nonempty complete subgraphs of $G$ and where $E_1$ consists of all
unordered distinct pairs in $V_1$, for which one is a subgraph of the other.
Denote by $G_m$ the successive Barycentric refinements of $G$ assuming $G_0=G$.
If $\lambda_0 \leq \dots \leq \lambda_n$ are the eigenvalues of the 
{\bf Kirchhoff Laplacian} $L=B-A$ of $G$, where $B$ is the diagonal {\bf degree matrix} and $A$ the
{\bf adjacency matrix} of $G$, define the {\bf spectral function} $F_G(x) = \lambda_{[n x]}$, where $[t]$ is the
largest integer smaller or equal to $t$. So, $F_G(0)=0$ and $F_G(1)$ is the largest 
eigenvalue of $G$. The $L^1$ norm of $F$ satisfies $||F_G||_1=||\lambda||_1/|V|={\rm tr}(L)/|V|
=2 |E|/|V|$ by the {\bf Euler handshaking lemma} telling that $2|E|$ is the sum of the vertex
degrees. In other words, $||F||_1=d (|V|-1)$, where $d$ is the {\bf graph density} 
$d=|E|/B(|V|,2)$ with Binomial coefficient $B(\cdot,\cdot)$.
The density $d$ gives the fraction of occupied graphs in the completed graph with vertex set $V$.
If $G,H$ are two sub graphs of some graph with $n$ vertices, define the {\bf graph distance}
$d(G,H)$ as the minimal number of edges which need to be modified to get
from $G$ to $H$. If $L,K$ are the Laplacians of $G,H$, then $\sum_{i,j} |L_{ij}-K_{ij}| \leq 4 d(G,H)$
because each edge $(i,j)$ affects the four matrix entries $L_{ij}, L_{ji},L_{ii},L_{jj}$ of the Laplacian
only: adding or removing that edge changes each of the 4 entries by $1$. 
The {\bf Lidskii-Last inequality} assures $||\mu-\lambda||_1 \leq \sum_{i,j=1}^{n} |A-B|_{ij}$
\cite{SimonTrace,Last1995} for any two symmetric $n \times n$ matrices $A,B$ with
eigenvalues $\alpha_1 \leq \alpha_2 \leq \dots \leq \alpha_n$ and
$\beta_1 \leq \beta_2 \leq \dots \leq \beta_n$. For two subgraphs $G,H$
of a common graph with $n$ vertices, and Laplacians $L,H$, the inequality gives
$||\lambda-\mu||_1 \leq 4 d(G,H)$ so that $||F_{G'}-F_{H'}|| \leq 4 d(G,H)/n$ if $G',H'$ are the graphs
with edge set of $G$ and vertex set of the host graph having $n(G,H)$ vertices. 
Therefore $||F_G-F_H||_1 \leq 4 d(G,H)/n(G,H)$ holds,
where $n(G,H)$ is the minimum of $|V(G))$ and $|V(H)|$ assuming both are in a common host graph.
The $L^1$ distance of spectral functions can so estimated by the graph distance.
Obviously, if $k$ disjoint copies of the same graph $H$ form a larger graph $G$, then 
$F(G)=F(H)$. Given a cover $U_j$ of $G$, let $H$ be graph generated by the set of vertices which 
are only in one of the set $U_j$. Let $K$ be the graph generated by the complement of $H$. 
Now, $d(G,H) \leq 4 |K|$ and $\sum_j d(U_j,H_j) \leq 4 |K|$. 
f $v_k$ is the number of complete subgraphs $K_{k+1}$ of $G$,
the {\bf clique number} of $G$ is defined to be $k$ if $v_{k}=0$ and $v_{k-1}>0$. 
A tree for example has clique number $2$ as it does not contain triangles. 
Denote by $\Gcal_d$ the class of graphs with clique number $d+1$. The class contains $K_{d+1}$ as
well as any of its Barycentric refinements. The {\bf clique data} of $G$ is the vector 
$\vec{v}=(v_0,v_1,\dots, )$, where $v_k$ counts the number of $K_{k+1}$ subgraphs of $G$. 
We have $v_0=|V|, v_1=|E|$ and $v_2$ counts the number of triangles in $G$. The clique data define
the {\bf Euler polynomial} $e_G(x) = \sum_{k=0}^{\infty} v_k x^k$ and the {\bf Euler characteristic}
$\chi(G) = e_G(-1)$. The polynomial degree of $e(x)$ is $d$ if $d+1$ is the clique number. 
There is a linear {\bf Barycentric operator} 
$A$ which maps the clique data of $G$ to the clique data of $G_1$. 
It is an upper triangular linear operator on $l^2$ with diagonal entries $A_{kk}=k!$. 
It also maps the Euler polynomial of $G$ linearly to the Euler polynomial of $G_1$. 
The {\bf unit sphere} $S(x)$ of a vertex $x \in V$ in $G$ is the graph generated by 
all vertices connected to $x$. The {\bf dimension} of $G$ is defined as ${\rm dim}(\emptyset)=-1$ and 
${\rm dim}(G)=1+\sum_{v \in V} {\rm dim}(S(x))/v_0$, where $S(x)$ is the unit sphere. A graph has
{\bf uniform dimension} $d$ if every unit sphere has uniform dimension $d-1$. The empty graph has 
uniform dimension $-1$. Given a graph $G$ with uniform dimension $d$, the {\bf interior} is the graph 
generated by the set of vertices for which every unit sphere $S(x)$ is a $(d-1)$-{\bf sphere}.
A $d$-{\bf sphere} is inductively defined to be a graph of uniform dimension $d$
for which every $S(x)$ is a $(d-1)$-sphere and for which removing
one vertex renders the graph contractible. This {\bf Evako sphere} definition starts with the assumption 
that the $(-1)$-sphere is the empty graph. {\bf Contractibility} for graphs is 
inductively defined as the property that there exists $x \in V$ for which both $S(x)$
and the graph generated by $V \setminus \{x\}$ are both contractible, starting with the assumption that $K_1$
is contractible. The {\bf boundary} $\delta G$ of a graph $G$ is the graph generated by the subset of 
vertices in $G$ for which the unit sphere $S(x)$ is not a sphere. Also the next definitions are inductive:
a {\bf d-graph} is a graph for which every unit sphere is a $(d-1)$-sphere;
a {\bf d-graph with boundary} is a graph for which every unit sphere is a $(d-1)$-sphere or $(d-1)$-ball;
a {\bf $d$-ball} is a $d$-graph with boundary for which the boundary is a $(d-1)$-sphere. 
Let $w_k(G_m)$ denote the number of $K_{k+1}$ subgraphs in the 
boundary $\delta G_m$. For $G=K_{d+1}$, the boundary of $G_m$ is a $(d-1)$-sphere for $m \geq 1$
and $\delta G_m$ contains the vertices for which the unit sphere in $G_m$
has Euler characteristic $1$. 
Barycentric refinements honor both the class of $d$-graphs as well as the class of $d$-graphs with boundary. 
Starting with $G=K_{d+1}$ which itself is neither a $d$-graph, nor a $d$-graph with boundary, 
the Barycentric refinements $G_m$ are all $d$-balls for every $m \geq 1$.
While the spectral functions $F$ are well suited to describe limits in $L^1([0,1])$, 
one can also look at the {\bf integrated density of states} $F^{-1}$, a monotone $[0,1]$-valued 
function on $[0,\infty)$ which is also called {\bf spectral distribution function} or {\bf von Neumann trace}.
It defines the {\bf density of states} $(F^{-1})'$ which is a probability measure on $[0,\infty)$
also called {\bf Plancherel measure}, analogue to the cumulative distribution functions
defining the law of the random variable which in absolutely continuous case is the probability
density function. Point-wise convergence of $F$ implies point-wise convergence
of $F^{-1}$ and so weak-* convergence of the density of states. 

\section{The theorem}

\begin{thm}[Central limit theorem for Barycentric subdivision]
The functions $F_{G_m}(x)$ converge in $L^1([0,1])$
to a function $F(x)$ which only depends on the clique number of $G$. 
The density of states converges to a measure $\mu$ which only depends
on the clique number. 
\end{thm}

We first prove a lemma which is interesting by itself. It allows to compute 
explicitly the clique vector of the subdivision $G_1$ from the clique vector of $G$. 

\begin{lemma}
There is an upper triangular matrix $A$ such that $\vec{v}(G_1) = A \vec{v}(G)$
for all finite simple graphs $G$.
\end{lemma}
\begin{proof}
When subdividing a subgraph $K_{k+1}$ it splits into $A_{kk} = (k+1)!$ smaller $K_{k+1}$ graphs. 
This is the diagonal element of $A$. Additionally, for $m>k$, every $K_{m+1}$ subgraph
produces $A_{km}$ subgraphs isomorphic to $K_{k+1}$, which is the number of interior $K_{k+1}$ subgraphs of the Barycentric
subdivision of $K_{d+1}$. These $K_{m+1}$ subgraphs in the interior correspond to $K_{m+2}$ subgraphs
on the boundary. This means that $A_{km}$ is the number of $K_{m+2}$ subgraphs of the boundary of
the Barycentric refinement of $K_{k+1}$. Since this boundary is of smaller dimension, we can use the
already computed part of $A$ to determine $A_{km}$. 
To construct $A$, we build up the columns recursively, starting to the left with $e_1$.
Let $B(n,k)=n!/(k! (n-k)!)$.
If $n$ columns of $A$ have been constructed, we apply the upper $n \times n$ part of $A$ 
to the vector $\vec{v}=[B(n+1,1),\dots,B(n+1,n)]^T$ in order to get
the clique data of the $(n-1)$-sphere $\delta (K_{n+1})_1$.
These numbers encode the number of interior complete subgraphs of the $n$-ball $(K_{n+1})_1$. If
the resulting vector is $[w_1,\dots,w_n]^T$, take $[1,w_1,\dots,w_{n-1},(n+1)!,0,0, \dots ]^T$ 
as the new column. The Mathematica code below implements this procedure. 
\end{proof}

The {\bf Barycentric operator} $A$ has an inverse $A^{-1}$ which is a bounded {\bf compact operator} on $l^2({\bf N})$.
Below we give the bootstrap procedure to compute $A$, adding more and more columns, using
the already computed $A$ to determine the next column. 
The eigenvalues $\lambda$ of $A$ are included in the spectrum $\sigma(A) = \{ 1!,2!,3!, \dots \}$. Any
eigenvector $f$ of $A^T$ can lead to {\bf invariants} $X(G) = \langle f,v(G) \rangle$
which correspond to {\bf valuations} in the continuum
as $\lambda X(G) = \langle f,v \rangle = \langle A^T f, v \rangle = \langle f,Av \rangle = X(G_1)$.
Of particular interest is the {\bf Euler characteristic eigenvector} $f=[1,-1,1,-1,....]^T$ to the 
eigenvalue $\lambda = 1$ verifying that $\chi(G)$ is an invariant under Barycentric subdivision. 
An other invariant is the eigenvectors $f=[0,\dots,0,-2,d+1]^T$ to the eigenvalue $d!$ which measures the 
``boundary volume" of a $d$-graph with boundary. For $d$-graphs, it remains zero under refinement. \\

Here is the proof of the theorem: 

\begin{proof}
{\bf (i)} {\bf For $G \in \G_d$, there are constants $C_d>0$ such that
$v_k(G_0) (d+1)!)^m/C_d \leq v_k(G_m) \geq v_k(G_0) C_d ((d+1)!)^m$ and
$w_k(G_0) (d!)^m/C_d \leq w_k(G_m) \geq w_k(G_0) C_d (d!)^m$ for every $0 \leq k \leq d$.}
Proof. This follows from the lemma and {\bf Perron-Frobenius} as the matrix $A$ is explicit. When 
restricting to $\G_d$, it is a $(d+1) \times (d+1)$-matrix with maximal eigenvalue $(d+1)!$,
when restricted to $G_{d-1}$ its maximal eigenvalue is $d!$. Diagonalizing $A_d=S_d^{-1} B_d S_d$
we could find explicit bounds $C_d$. \\

{\bf (ii)} {\bf For every $G=K_k$, the sequence $F_{G_m}$ converges. }
Proof: There are $(d+1)!$ subgraphs of $G_{M=1}$ isomorphic to $G_m$, forming a 
cover of $G_{m+1}$. They intersect in a lower dimensional graph. Since the number of 
vertices of this intersection grows with an upper bound $C_d d!^m$ in each $G_m$ 
we have $d(G_{m+1},\bigcup_j G_m^{(j)} ) \leq (d+1)! C_d (d!)^m$. This shows 
$||F_{G_m}-F_{G_{m+1}}||_1 \leq C_d (d+1)! (d!)^m/((d+1)!)^m = C_d/(d+1)^{m-1}$. 
We have a Cauchy sequence in $L^1([0,1])$ and so a limit in that Banach space. \\

{\bf (ii)} {\bf Barycentric refinement is a contraction on $\Gcal_d$ in an adapted metric}.
Proof. Every $G \in \Gcal_d$ can be written as a union of several $H_m \sim K_{d+1}$ subgraphs and 
a rest graph $L$ in $\Gcal_{k}$ with $k<d$. Then $F_{G_m}$ and $F_{H_m}$ have the same
limit because $F_{H_m}=F_{G_{m-1}}$.
The Barycentric evolution of the boundary of refinements of $K_{d+1}$ as well as $L$ grows 
exponentially smaller. Given two graphs $A,B \in \Gcal_d$, then 
$d(A_m,B_m) \leq C_d (d!)^m$ so that $||F_{A_m}-F_{B_m}||_1 \leq 4 C_d/(d+1)^m$. Let $m_0$ be so
large  that $c=C'/(d+1)^{m_0}<1$. Define a new distance $d'(A,B) = \sum_{k=0}^{m_0-1} d(A_k,B_k)$,
so that $d'(A_1,B_1) \leq c d'(A,B)$. Apply the {\bf Banach fixed point theorem}.  \\

{\bf (vi)} {\bf There is uniform convergence of $F_{G_m}$ on compact intervals of $(0,1)$}.
Since each $F_{G_m}$ is a monotone function in $L^1([0,1])$,
the exponential convergence on compact subsets of $(0,1)$ follows from 
exponential $L^1$ convergence. (This is a general real analysis fact \cite{LewisShisha}). 
In dimension $d$, one has $||F_{G_m}||_1 \to (d+1)!$ exponentially fast.
Indeed, as the number of boundary simplices grows like $(d!)^m$ and the number
of interior simplices grows like $((d+1)!)^m$, the convergence is of the order  $1/(d+1)^m$.
By the {\bf Courant-Fischer mini-max principle}
$F_{G}(1) = \lambda_{n} = {\rm max} (v,Lv)/(v,v) \geq {\rm max}(L_{xx}) = {\rm max}({\rm deg}(x))$
grows indefinitely, so that the $L^{\infty}$ convergence can not be extended to $L^{\infty}[0,1]$. 
\end{proof} 

It follows that the density of states of $G$ converges ``in law" to a universal density of states 
which only depends on the clique number class of $G$, hence the name ``central limit theorem".
The analogy is to think of $G$ or its Laplacian $L$ as the random variable and the spectrum 
$\sigma(L)$ of the Laplacian $L$ as the analog of the probability density and of the
Barycentric refinement operation as the analogue of adding and normalizing two independent
identically distributed random variables. 
For $d=1$, where the graph is triangle-free and contains at least one edge, we know everything:

\begin{propo}
For $d=1$, the limiting function is $F_1(x) = 4 \sin^2(\pi x/2)$.
\end{propo}
\begin{proof}
As the limiting distribution is universal, we can compute it for $G=C_m$, where
$G_m = C_{4 \cdot 2^m}$. As the spectrum of $C_n$ is the set $\{ 4 \sin^2(\pi k/n) \; | \; k=1, \dots, n\}$,
the limit is $F_1$. 
\end{proof}

For $d=1$, the limiting spectral function is related to the {\bf Julia set} of the
{\bf quadratic map} $z \to z(4-z)$ which is conjugated to 
$z \to z^2-2$ ($c=-2$ at bottom tail of Mandelbrot set) or $z \to 4z(1-z)$ 
which is the {\bf Ulam interval map} conjugated to {\bf tent map}, or $z \to 4z^2-1$ 
which is a {\bf Chebyshev polynomial}.

\section{Remarks}

{\bf 1)} For $d=2$ already, we expect spectral gaps. A first large one is observed at 
$x=1/2$. For $G_4$ with $G=K_3$, we see a jump at $0.5$ of $2.002$, 
Starting with $G=K_3$, the graph $G_2$ has $25$ vertices. 
Its Laplacian has eigenvalues for which
$\lambda_{13}-\lambda_{12} =2.0647...\sim 2$ being already close to the gap. 
The eigenvalues can be simplified to be roots of polynomials of degree $4$
for which explicit radical expressions exist allowing to estimate.
We also know $||F_{G_3}- F_2||_1 \leq 8 \sum_{k=3}^{\infty} 1/3^k = 8/18<1/2$.
While we know that $F_{G_m}- F_2$ converges pointwise and uniformly on each 
compact interval we don't have uniform constant bounds on each interval.\\
{\bf 2)} Instead of the Laplacian $L=B-A$, we could take the {\bf adjacency matrix} $A$.
We could also take the {\bf multiplicative Laplacian} $L'=A B^{-1}$ or the to $L'$ isospectral
selfadjoint {\bf random walk Laplacian} $L''=B^{1/2} A B^{1/2}$. The corresponding 
spectral functions always converge exponentially fast, but to different 
limiting functions. For the adjacency matrix for example, the spectral gaps appear much smaller.  \\
{\bf 3)} The convergence also works for the {\bf Dirac operator} $D=d+d^*$, where $d$ is
the {\bf exterior derivative} or the {\bf Hodge Laplacian} $L=D^2$.
For the Dirac operator $D$, for which the density of states is supported on $(-\infty,\infty)$.
For the Hodge Laplacian on $[0,\infty)$. The Hodge Laplacian factors into different form 
sectors $L_k$ and the spectral functions of each 
{\bf form Laplacian} $L_k$ converge. The blocks which make up $D^2$.
The {\bf scalar Laplacian} $L_0 = d^* d$
agrees with the {\bf combinatorial Laplacian} = Kirchhoff Laplacian $L=B-A$ discussed above. \\
{\bf 4)} Any Barycentric refinement preserves the {\bf Euler characteristic} $\chi(G)=$ 
$\sum_{k=0} (-1)^k v_k$. We know that $G_n$ is homotopic and even homeomorphic to $G$.
$G$ and $H$ are {\bf Ivashchenko homotopic} \cite{I94,CYY} if one get $H$ from $G$ by 
homotopy transformation steps done by adding a vertex, connecting it to a 
contractible subgraph, or removing one with a contractible sphere. A topology $\O$ on $V$ of $G$ 
is a {\bf graph topology} if there is a sub-base $\B$ of $\O$ consisting of
contractible subgraphs such that the intersection of any $A,B \in \B$ satisfying 
$\dim(A \cap B) \geq {\rm min}(\dim(A),\dim(B))$
is contractible, and every edge is contained in some $B \in \B$.
We ask the $\G$ of $\B$ to be homotopic to $G$, where the {\bf nerve graph} 
$\G = (\B,\E)$ has edges $\E$ consisting of all pairs $(A,B) \in \B \times \B$ for which
the dimension assumption is satisfied. 
A map between two graphs equipped with graph topologies is {\bf continuous}, if it is a graph 
homomorphism of the nerve graphs such that ${\rm dim}(\phi(A)) \leq {\rm dim}(A)$ for every $A \in \B$.
If $\phi$ has a continuous inverse it is a {\bf graph homeomorphism}. \\
{\bf 5)} We know ${\rm dim}(G_1) \geq {\rm dim}(G)$ with equality for $d$-graphs, graphs for which 
unit spheres are spheres. For a visualization of the discrepancy on random Erd\"os-Renyi graphs, 
see \cite{KnillProduct}. We also know that ${\rm dim}(G_m)$ converges monotonically to the 
dimension of the largest complete subgraph because the highest dimensional simplices dominate
exponentially.   \\
{\bf 6)} If $G$ is a $d$-graph, then for $m>1$, all $G_m$ as well as any finite intersections of 
unit spheres in $G_m$ are all {\bf Eulerian} so that we can define a {\bf geodesic flow} on $G_m$
or a {\bf billiard} on Barycentric refinements of the ball $(K_k)_m$. 
For non-Eulerian graphs like an icosahedron, a light propagation on the vertex set is not defined 
without breaking some symmetry. Also, for any graph and $m>0$, 
the {\bf chromatic number} of $G_m$ is the clique number of $G$. 
Indeed, for $m \geq 1$ the dimension of the complete subgraph graph $x$ of $G_m$ can be taken as the
{\bf color} of the vertex $x$ in $G_{m+1}$. Since the dimension takes values in $\{ 0, \dots, d\}$ the 
chromatic number of $G$ agrees with the clique number $d+1$ of $G$. \\
{\bf 7)} If $d_0 \leq \cdots \leq d_n$ is the {\bf degree sequence} of $G$, define 
the {\bf degree function} $H_G(x) = d_{[x n]}$ in the same way as the spectral function.
Since the degrees are the diagonal elements of $L$,
the {\bf integrated degree function} $\tilde{H}(x) =  \int_0^x H_{G}(t) \; dt$
and the {\bf integrated spectral function} $\tilde{F}(x) = \int_0^x F_G(t) \; dt$
agree at $x=0$ and $x=1$. By the {\bf Schur inequality}, we have
$\tilde{H}(x) \geq \tilde{F}(x)$. The degree function $H_{G_m}$ also converges. \\
{\bf 8)} Barycentric refinements are usually defined for realizations
of a {\bf simplicial complex} in Euclidean space. The Barycentric subdivision of a graph is
not the same than what is called a {\bf simplex graph} \cite{BandeltVandeVel}
as the later contains the empty graph as a node. It also is not the same than the
{\bf clique graph} \cite{Hamelink68} as the later has as vertices the maximal complete 
subgraphs, connecting them if they have a non empty intersection. 
The Barycentric refinement is the {\bf graph product} of $G$ with $K_1$ \cite{KnillProduct}. 
The product $G \times H$ has as vertex set the set of all pairs $(x,y)$, where $x$ 
is a complete subgraph $x$ of $G$ and a complete subgraph $y$ of $H$. 
Two such vertices $(x,y)$ and $(u,v)$ are connected by an edge, if 
either $x \subset u$ and $y \subset v$ or $u \subset x$ and $y \subset v$.  \\
{\bf 9)} The limit of Barycentric refinements is ``{\bf flat}" in the following sense:
the graph {\bf curvature} $K(x)=1-V_0(x)/2+V_1(x)/3- \cdots$ with $\vec{V}(x)=(V_0(x),V_1(x),\dots)$ 
being the clique data of the unit sphere $S(x)$ is defined for all finite simple graphs \cite{cherngaussbonnet} 
satisfying the {\bf Gauss-Bonnet-Chern theorem} adds up to the Euler characteristic of the original graph. 
Because the total curvature is $1$, and as the graph gets larger, the curvature 
averaged over some subgraph $G_{m-k}$ of $G_m$ goes to zero,
while the individual curvatures can grow indefinitely. For $d=2$ for example, where the
curvature is $K(x)=1-d(x)/6$, there are vertices with very negative curvature, but 
averaging this over a smaller patch gives zero. In general, if we look at subgraphs of $G$
which are wheel graphs as $2$-dimensional sections, also the {\bf sectional curvatures} become
unbounded at some points but averages of sectional curvatures over two dimensional surfaces
obtained by Barycentric refinements of a wheel graph go to zero. The limiting holographic object
looks like Euclidean space. There is even {\bf ``rotational symmetry"}, not everywhere, but with centers
at a dense subset of the continuum. \\
{\bf 10)} The graph theoretical subdivision definition uses the point of view that  
complete subgraphs of a graph are treated as {\bf points}. When this identification is taken
seriously and iterated, we get a {\bf holographic Barycentric refinement sequence} which
is already realized in the graph itself. This is familiar when building up number systems: if the set of 
integers $G=Z$ is considered to be a graph with adjacent integers connected, 
the Barycentric refinements $G_n$ contain all {\bf dyadic rationals} $k/2^n$.
Modulo $1$, this is a {\bf Pr\"ufer group} $\PP_2$ which has as the {\bf Pontryagin dual} the compact 
topological group $\D_2$ of dyadic integers, a subring of the {\bf field} $\Q_2$ of {\bf 2-adic numbers}.
In the case $d=1$, there is a limiting random operator on the group of {\bf dyadic integers}.  \\
{\bf 11)} There is an analogy between $(\R,\T,\Z, x \to x+\alpha \; {\rm mod} \; 1,x \to 2x \; {\rm mod} \; 1)$
and $(\Q_2,\PP_2,\D_2, x \to x+1,\sigma)$, where $x \to x+1$ is the addition on the compact topological group. 
This translation on $\D_2$ is called {\bf adding machine} \cite{Friedman} and $\sigma$ is shift which is 
a return map on half of $\D_2$. There is a natural way to get the group $\D_2$ through {\bf ergodic theory} as it is
the unique fixed point of 2:1 {\bf integral extensions} in the class of all dynamical system. It is
called also the {\bf von Neumann-Kakutani system} and usually written as an interval map on $[0,1]$. 
The ergodic theory of systems with discrete spectrum is completely understood \cite{CFS,DGS} and
the von-Neumann-Kakutani system belongs so to a class of systems, where we can solve the {\bf dynamical log-problem}: 
given any $x,y$ and $\epsilon>-$ find $n$ such that $d(T^n x,y) < \epsilon$ which is essential 
everywhere in dynamics, from prediction of events up to finding solutions to Diophantine equations.
Such systems are typically {\bf uniquely ergodic} 
and so by spectral theory naturally conjugated to a group translation on a compact topological group,
the dual group of the eigenvalues of the {\bf Koopman operator} on the unit circle in the complex plane. 
There is an important difference between the real and 2-adic story: while in both cases, the chaotic scaling 
systems are isomorphic {\bf Bernoulli systems} on compact Abelian topological groups, the group translation on 
the 2-adic group is naturally unique and {\bf quantized}, while on the circle, there are many group 
translations $x \to x+ \alpha$. In the {\bf real picture}, there is a continuum of natural translations, in the 
{\bf dyadic picture}, there is a {\bf smallest translation}. The picture is already
naturally quantum. Egyptian dyadic fractions, music notation or
the transition from Fourier to wavelet theory can be seen as a move towards {\bf dyadic models}. 
See \cite{TaoDyadic} for a plethora of other places. \\
{\bf 12)} For a $d$-graph $G$ with $d>1$, the renormalization limit of $G$ is
still unidentified. The limiting operator is likely a random operator on a compact topological
group. As the group of dyadic integers, it is likely also a {\bf solenoid}, an 
inverse limit of an inverse system of topological groups. There is a group acting on it producing
the operator in a crossed product $C^*$-algebra. This is at least the picture in one dimensions. \\
{\bf 13)} As $F_{G_m}$ converges in $L^1$, we have convergence of the 
{\bf integrated density of states} $F_{G_m}^{-1}$ in $L^1[0,\infty)$ and so
weak-* convergence of the derivative, the {\bf density of states}, in the Banach space of measures. 
In the case $d=1$, the density of states is $f(x) = (x(4-x))^{-1/2}/\pi$ supported on $[0,4]$.
The integrated density of states is $F^{-1}(x) = (2/\pi) \arcsin(\sqrt{x}/2)$.
The measure $\mu = f(x) 1_{[0,4]} dx$ is the {\bf equilibrium measure} on the Julia set of
$T(z) = 4z-z^2$ whose dynamic is conjugated to $z \to z^2-2$ in the 
Mandelbrot picture or the {\bf Ulam map} $z \to 4x(1-x)$ for maps on the interval $[0,1]$.
The spectral function $F$ satisfies $T(F(x))=F(2x)$.
The measure $\mu$ maximizes metric entropy and equals it to topological entropy 
$\log(2)$ of $T$. \\
{\bf 14)} In the case $d=2$, the numbers $v_k$ were already known \cite{Snyder,hexacarpet}
as $v_2(G_m)=6^m$ and $v_0$ is the sum of $v_0(m-1),v_1(m-1),v_2(m-1)$
and $v_0-v_1+v_2=2$ leading to formulas like $v_0(G_m) = 1-3(2^{m-1}+2^m+2^{m-1} 3^{m-1})$,
$v_1(G_m) = 3(-2^{m-1}+2^m+2^{m-1} 3^m)$. As the lemma shows, in general, the clique data of $G_1$ are 
a linear image of the clique data of $G$ with a linear map $A$ independent of $G$. 
Since $A$ has a compact inverse, we can look at an eigenbasis of $A$ and could write down
explicit formulas for the number of vertices of $G_n$, if the initial clique data of $G$ are known. \\
{\bf 15)} If $H$ is a subgraph of $G$, then each refinement $H_m$ is a subgraph of $G_m$. 
Also, $G_k$ is a subgraph of $G_m$ if $k \leq m$ and
the automorphism group of $G_m$ contains the automorphism group of $G$.
The case $G=K_n$ shows that the {\bf automorphism group} can become larger. In general,
${\rm Aut}(G_1)={\rm Aut}(G)$. It is only in 
rare cases like $G=C_n$ that ${\rm Aut}(G_m)$ can grow indefinitely. \\
{\bf 16)} The matrix $A$ mapping the clique data of $G$ to the clique data of $G_1$
is a special case of the following: for any graph $H$, there is a linear map $A_H$
which maps the clique data of any graph $G$ to the clique data of $G \times H$. 
This operator $A_H$ depends only on $H$. While $A=A_{K_1}$ is invertible on each 
class of finite dimensional graphs. The spectral data $v(G)$ do not determine the 
graph $G$ as trees already show, the Barycentric refinement operator $T$ is invertible
on the image $G(\Gcal)$ of the class $\Gcal$ of all finite simple graphs. \\
{\bf 17)} Barycentric refinements make graphs nicer: the chromatic number is the clique number, 
the graphs are Eulerian. If we identify graphs with their
Barycentric refinements, we have a {\bf holographic picture}. There is a natural
geodesic flow on $G_1$ if $G$ is a $d$-graph as each every unit sphere has a 
natural involution, defined inductively by running the geodesic flow on the unit
sphere. On a two dimensional sphere for example, the flow is defined because every
vertex has even degree. This allows to draw {\bf straight lines} on it and to define an
antipodal map. For a 3-sphere, because every unit sphere has such an antipodal map, we can 
define straight lines there, in turn leading to an antipodal map on such spheres. This can be 
continued by climbing up the dimensions to get a geodesic flow on the graph $G_1$ itself. \\
{\bf 18)} The density of states obtained as a limiting spectral density of finite dimensional
situations is central in the theory of {\bf random Schr\"odinger operators} $L$ 
\cite{Cycon,Carmona,Pastur}, where the density of states
$\mu$ can be defined as the functional $f \to {\rm tr}(f(L))$ as $L$ is an element in a 
von Neumann algebra with trace, the crossed product of $L^\infty(\Omega)$ with an
ergodic group action on the probability space. The {\bf Birkhoff ergodic theorem} has been 
applied by Pastur to assure in that theory that the density of states of finite dimensional operators 
defined by orbits starting at some point $x$ converges {\bf almost surely} to the measure $\mu$. 
The $1$-dimensional discrete case $d=1$, where a single ergodic transformation $T$ on a probability 
space is given is one of the most studied models, 
in particular the map $x \to x+ \alpha \; {\rm mod} \; 1$ leading to almost
periodic matrices like the {\bf almost Mathieu operator}. 
It is the real analog of the limiting operator which is almost periodic over the 
dyadic integers.
There is no doubt that also in higher dimensions, the limiting model of Barycentric refinement 
is part of the theory of 
aperiodic almost periodic media \cite{Pastur,Cycon,Pastur,Senechal,Simon82,KellendonkLenzSavinien}.

\section{Open ends}

{\bf 1)} Already for $d=2$, we do not know what the nature of $F_d$ is. One must suspect that
there is a renormalization picture, as in one dimensions. It can not be a Julia set of a 
polynomial, as the spectrum becomes unbounded. We suspect that for $d>1$, the time of the 
group action is higher dimensional and that the limiting Laplacian remains almost periodic. 
Since we can not compute $F_2(G_m)$ yet for large $m$ due to the fast growth of $G_m$, we see
experimentally that the function $F_2$ on $L^1([0,1])$ has a {\bf self-similar nature} in the sense
that $F_2(6x) \sim  \phi_2(F_2(x))$ for some $\phi_2$ and $x \in [0,1/6]$.
This self-similar nature is present in $d=1$ as 
$F_1(2x) = \phi_1(F_1(x))$, where $\phi_1$ is the {\bf quadratic map} $\phi_1(x) = x(4-x)$. 
Since $F_2(1)$ is infinite, the map $\phi_2$ (if it exists at all) must be an unbounded 
function and in particular can not be a polynomial. \\

{\bf 2)} In order to establish an existence of a {\bf spectral gap} for $F_2$, 
one could first establish an exact finite approximation bound, then estimate the rest. 
For $d=2$, the first large gap opens at $x=1/2$. This could be related to the fact that 
every step doubles the vertex degrees in two dimensions. Here are the largest spectral jump
values for $G_5$ if $G=K_3$: $\Delta(0.5)=2.002..$, $\Delta(0.835366)=1.67669...$
$\Delta(0.917683)=2.86249$, $\Delta(0.972561)=3.89394$, $\Delta(0.98628)=6.93379$
$\Delta(0.995427)=7.96794$ and $\Delta(0.998476)=14.9767$. To investigate gaps, we would 
need a Lidskii type estimate for the smaller half of the eigenvalues. \\

{\bf 3)} We see experimentally that for $d=2$, the sequence $F_{G_m}$ appears monotonically 
increasing $F_{G_{m}} \leq F_{G_{m+1}}$ at least for the small $m$ in which we can compute it.
In the case $d=1$, the monotonicity is explicit. 
It fails for smaller $m$ and $d=3$ but could hold for any $d$ if $m$ is large enough.
If we had more details about the convergence in the middle of the spectrum, we could attack
the problem of verifying that spectral gaps exist. \\

{\bf 4)} Unlike the {\bf Feigenbaum renormalization} in one dimensional dynamics 
\cite{Feigenbaum1978,Lanford84,DeMelo}, where a hyperbolic attractor with stable and unstable 
manifold exists, the Barycentric renormalization is a contraction and the convergence
proof is more elementary. More so than the {\bf central limit theorem} in probability
theory, where the renormalization map is
$X \to \overline{X_1+X_2}$, where $X_i$ are independent random variables on a probability 
space $\Omega$ with the same distribution
than $X$ and $\overline{X}=(X-{\rm E}[X])/\sigma(X)$. That random variable 
renormalization is not a uniform contraction in $L^2(\Omega,{\rm P})$. 
For $d=1$ there is a limiting almost periodic operator where the compact topological 
{\bf group of dyadic integers} becomes visible in the Schr\"odinger case. We expect also in higher 
dimensions, a random limiting operator and that the renormalization
map is robust in the sense that one can for example shift the energy $E$ and 
have a deformed attractor. In the one-dimensional Schroedinger case, where one deals with 
Jacobi matrices, one gets so to spectra on Julia sets $J_E$ of the quadratic map 
$T(x) = z^2+E$ \cite{BGH,BBM,Kni93a,Kni93b,Kni95}. 
Also in higher dimensions, we expect that a Schr\"odinger renormalization picture reveals the
underlying topological group if an energy parameter is used to modify the renomalization. \\

{\bf 5)} In the case $d=1$, the roots of the {\bf Dirac zeta function} 
$\zeta(s) = \sum_{\lambda>0} \lambda^{-s}$, defined by the positive eigenvalues
$\lambda$ of the Dirac operator $D=d+d^*$ \cite{KnillILAS,McKeanSinger} of $G_{m}$,  
converges to a subset of the line ${\rm Re}(s)=1$. What happens in the case $d=2$?
Already for $d=1$, the convergence of the roots is slow of the order $\log(\log(v_0(G))$
as it is initially proportional to $m$ and slowing down exponentially when approaching the line.
Any experimental investigations in $d=2$ would be difficult as no explicit formulas for the 
eigenvalues exist. \\

{\bf 6)} It would be nice to know more about the linear operators $A_H$ which 
belong to the map $G \to G \times H$ on $\Gcal$. Their inverse is compact. 
In the case $H=K_1$, where we get
the Barycentric refinement operator $A$, the eigenvector of $A^T$ to the eigenvalue $1$
gives an invariant for Barycentric refinement: it is the Euler characteristic vector
$(1,-1,1,-1,1,-1, \dots)$. When restricted to $d$-graphs, the vector $(0,0,...,1)$
with $1$ at the $d+1$'th entry is an eigenvector to the eigenvalue $(d+1)!$. It leads
to a counting invariant in the limit which is called {\bf volume}. The other eigenvectors
lead to similar limits which must be linear combinations of {\bf valuations} in Riemannian geometry
\cite{Santalo,KlainRota}. We see that the even {\bf Barycentric invariants} 
$X_{2k}(G) = \vec{f}_{2k}(A) \cdot \vec{v}(G)$ are constant zero on $d$-graphs, where
$\vec{f}_k(A)$ are the eigenvectors of $A^T$. It will imply for example that for any 
compact smooth 4-manifold containing $v$ vertices, $e$ edges, $f$ triangles, $g$ tetrahedra,
and $h$ pentatopes satisfies $22e+40g=33f+45h$. 
In the continuum, the more general operator $A_H$ leads to a linear map on valuations. 
These maps will help to investigate the connection between the discrete and continuum.  \\

{\bf 7)} In the same way as for Riemannian manifolds, where one studies the eigenfunctions
$f_k$ of the Laplacians, the eigenfunctions of the Kirchhoff Laplacians will play an important
role for understanding the limiting density of states.
The  {\bf Chladni figures} are the level sets $f_k=0$. 
They can be defined for $d$-graphs. We can look at $f=0$. It is by a discrete Sard result \cite{KnillSard}
a $(d-1)$-graph as long as $0$ is not a value taken by $f$. If $0$ is a value taken we
just can plot $f=\epsilon$ for $0<\epsilon$ small. The point is that we can so define
nice $(d-1)$ graphs associated to $f_k$. The topology of these {\bf Chladni graphs} depends
very much on the energy as in the continuum and as in the continuum, are not well understood 
yet. \\

{\bf 8)} The even eigenvectors $\vec{v}_{2k}$ of $A^T$ appear to have the property that 
$X_{2k}(G) = \vec{v}_{2k} \cdot \vec{v}(G)=0$ for any $2d$-graph. If true, this leads 
to integral geometric invariants in the limit. For any graph, define 
$X_{2k}(G) = \lim_{m \to \infty} \vec{v}_{2k} \vec{v}(G_m)/\lambda_{2k}^{-m}$. The limit is trivial as the right
hand side is constant. For $d$-graphs, we see $X_{2k}(G)=0$. We have tried random versions of 
$S^4,S^2 \times S^2, T^4,S^2 \times T^2, S^3 \times T^1,S^6$.
The limiting integral theoretic invariants for $d$-manifolds would be zero for differentiable manifolds.
Assume we would find a graph $G$ with uniform dimension $d$ which is not a $d$-graph but which is homeomorphic     
to a $d$-graph $H$. Then, since Barycentric refinements preserve the homeomorphism relation,
the Barycentric limit $M$ of $G$ is a topological manifold homeomorphic to the Barycentric
limit $N$ of $H$. While $M,N$ are topological manifolds which are homeomorphic, 
they can not be diffeomorphic, as integral geometric integer valued invariants are diffeomorphism invariants.
A basis for valuations can be obtained if $M$ is embedded in a projective
$n$-sphere. We can compute the expectation of the $k$-volume of a random $m$-planes with $M$
using a natural probability measure obtained from Haar measure on $SO(n)$ acting
on $k$-planes. This leads to $d+1$ invariants, for which some linear combinations is the 
Euler characteristic. The graph theoretical invariants obtained from $A$ could produce invariants allowing
to distinguish homeomorphic but not diffeomorphic manifolds.

\section{About the literature} 

Barycentric refinement are of central importance 
in topology. Dieudonn\'e \cite{Dieudonne1989} writes of 
{\it "the three essential innovations that launched 
combinatorial topology: simplicial subdivisions by the barycentric method, the use
of dual triangulation and, finally, the use of incidence matrices and of their
reduction."} In algebraic topology, it is used for
proving the {\bf excision theorem} or the {\bf simplicial approximation
theorem} \cite{Hatcher,Rotman}. In topology, Barycentric subdivision is primarily used in 
an Euclidean setting for subdividing complex polytops or CW complexes.
For abstract simplicial complexes, it is related to flag complexes even so
there are various inequivalent definitions of what simplicial subdivision or simplex
graph or Barycentric subdivisions are. In graph theory, subdivisions
are considered for simplicial complexes which are special {\bf hypergraphs} in topological
graph theory \cite{TuckerGross}. Indeed, most graph theory treats graphs as 
one-dimensional simplicial complexes, ignoring the {\bf Whitney complex} of all complete 
subgraphs. Subdivisions classically considered for graphs only agree with the definition used here if
the graph has no triangles. Two graphs are {\bf classically homeomorphic}, if they have isomorphic
subdivisions ignoring triangles (see e.g. \cite{Bollobas1,HHM,BR}). In the context of {\bf maps} which 
are finite cell complexes whose topological space is a surface $S$, Barycentric subdivisions
are considered for this cell complex \cite{handbookgraph} but not for the graph.
In the context of convex polytops, Barycentric subdivision appear for the {\bf order complex}
of a polytop which is an abstract simplicial complex \cite{McMullenSchulte}. \\

Resources on the spectral theory of graph are in 
\cite{Chung97,Mieghem,Brouwer,VerdiereGraphSpectra,Post,BLS}. 
It parallels to a great deal the corresponding theory for Riemannian manifolds
\cite{Chavel,Rosenberg,BergerPanorama}. \\

The definition of $d$-spheres and homotopy are both due to Evako. Having developed the 
sphere notion independently in \cite{KnillEulerian,KnillProduct} we realized in
\cite{KnillJordan} the earlier definition of Ivashchenko=Evako
\cite{I94a,I94,Evako1994,Evako2013}. The definition of these Evako spheres is 
based on Ivashchenko homotopy which is homotopy notion inspired by Whitehead 
\cite{Whitehead} but defined for graphs. \\

The Barycentric refinements for $d=2$ are studied in \cite{hexacarpet}, where the 
limit of Barycentric refinement has a dual called the {\bf hexacarpet}. 
Figure~{7)} in that paper shows
the eigenvalue counting function in which gaps in the eigenvalues are shown similarly
as in the spectral function $F_G(x)$. The work \cite{hexacarpet} must therefore 
be credited for the experimental discovery of the gap.  \\

Papers like \cite{DiaconisMiclo,Hough} deal with 
the geometry of tessellations of a triangle, which defines a random walk leading to 
a dense subgroup of $SL(2,R)$ which defines a Lyapunov exponent.
While the focus of those papers is different, there might be relations. \\

For illustrations, more motivation and background, see also our first 
write-up \cite{KnillBarycentric}.

\section{Figures and Code}

\begin{figure}
\scalebox{0.25}{\includegraphics{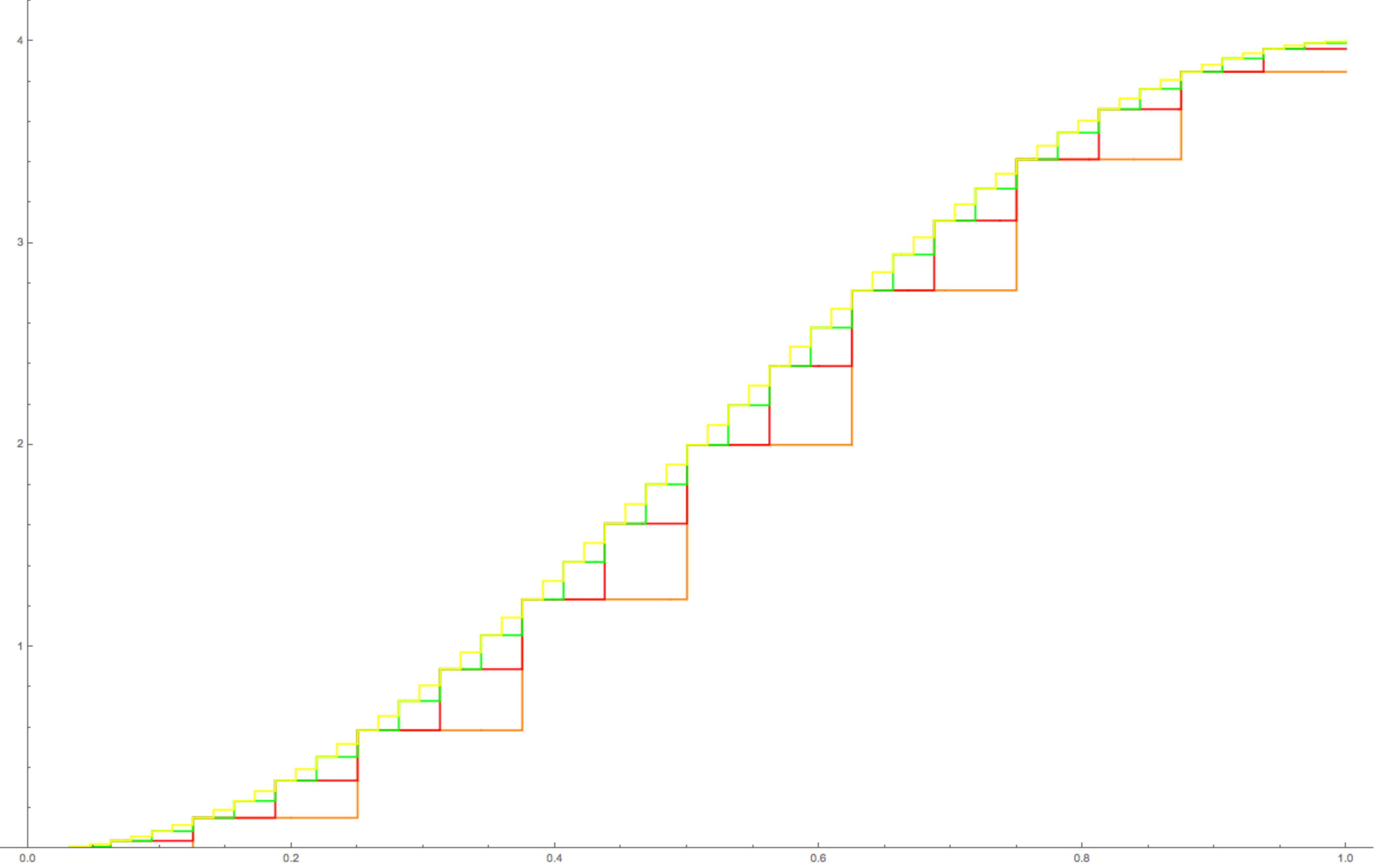}}
\scalebox{0.25}{\includegraphics{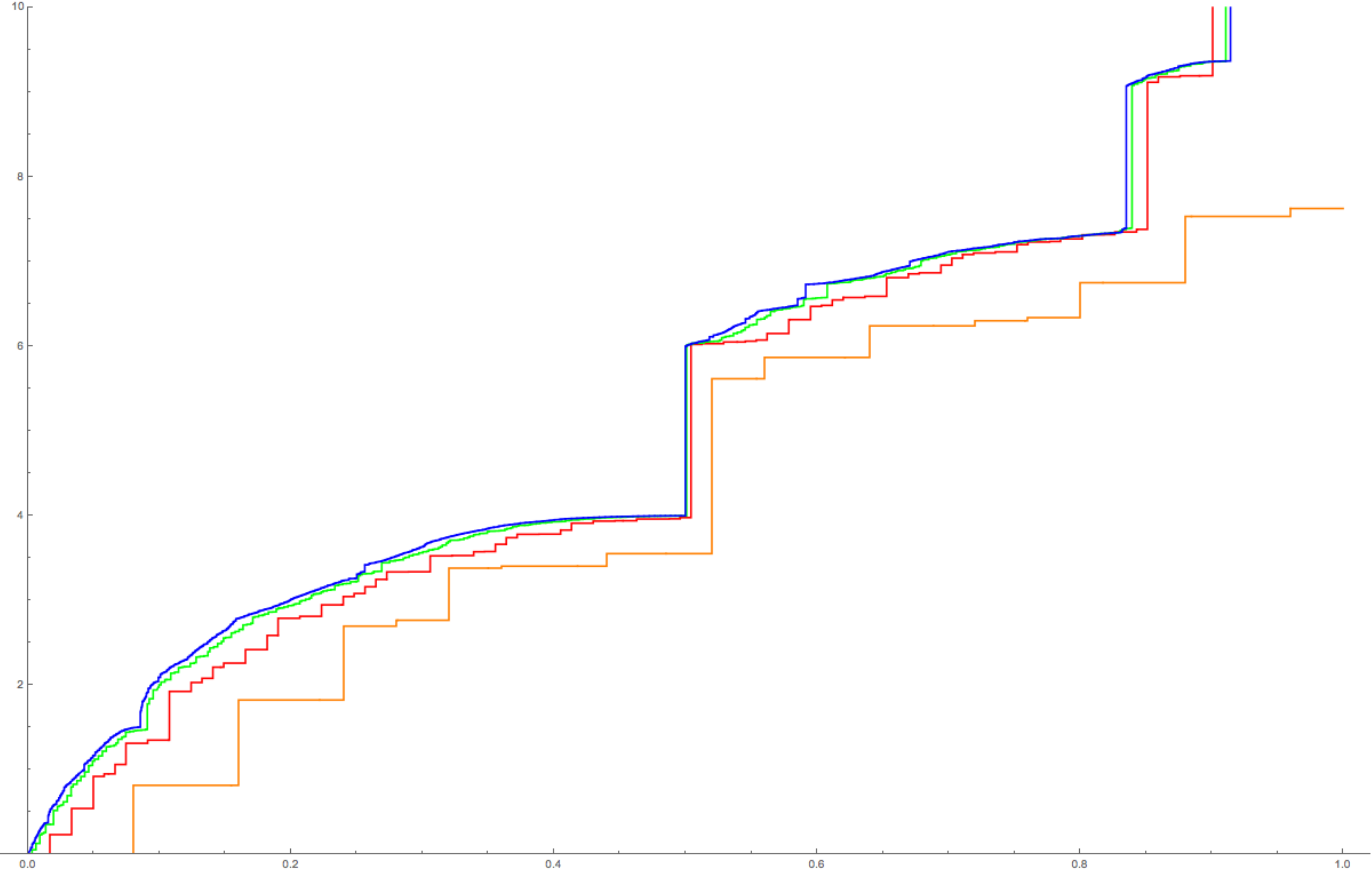}}
\caption{
The functions $F_{G_m}$ for $G=C_4$ converge to the limiting function $4 \sin^2(\pi x/2)$. 
The functions $F_{G_m}$ for $G=K_3$ converge to a limiting function which appears to have 
jumps corresponding to gaps in the spectrum. We have here only established that the universal
limit $F$ exists, but already for $d=2$ do not know about the nature of the limit. 
}
\end{figure}

Here are the Mathematica routines which produced the graphs: 

\vspace{12mm}

\begin{tiny}
\lstset{language=Mathematica} \lstset{frameround=fttt}
\begin{lstlisting}[frame=single]
TopDim=2; 
Cl[s_,k_]:=Module[{n,t,m,u,q,V=VertexList[s],W=EdgeList[s],l},
 n=Length[V]; m=Length[W]; u=Subsets[V,{k,k}]; q=Length[u]; l={};
 W=Table[{W[[j,1]],W[[j,2]]},{j,m}];If[k==1,l=Table[{V[[j]]},{j,n}],
 If[k==2,l=W,Do[t=Subgraph[s,u[[j]]]; If[Length[EdgeList[t]]==
 Binomial[k,2],l=Append[l,VertexList[t]]], {j,q}]]];l];
Ring[s_,a_]:=Module[{v,n,m,u,X},v=VertexList[s]; n=Length[v];
 u=Table[Cl[s,k],{k,TopDim+1}] /. Table[k->a[[k]],{k,n}];m=Length[u];
 X=Sum[Sum[Product[u[[k,l,m]],
   {m,Length[u[[k,l]]]}],{l,Length[u[[k]]]}],{k,m}]];
GR[f_]:=Module[{s={}},Do[Do[If[Denominator[f[[k]]/f[[l]]]==1 && k!=l,
 s=Append[s,k->l]],{k,Length[f]}],{l,Length[f]}];
 UndirectedGraph[Graph[s]]];
GraphProduct[s1_,s2_]:=Module[{f,g,i,fc,tc},
 fc=FromCharacterCode; tc=ToCharacterCode;
 i[l_,n_]:=Table[fc[Join[tc[l],IntegerDigits[k]+48]],{k,n}];
 f=Ring[s1,i["a",Length[VertexList[s1]]]];
 g=Ring[s2,i["b",Length[VertexList[s2]]]]; GR[Expand[f*g]]];
NewGraph[s_]:=GraphProduct[s,CompleteGraph[1]];
Bary[s_,n_]:=Last[NestList[NewGraph,s,n]];
Laplace[s_,n_]:=Normal[KirchhoffMatrix[Bary[s,n]]];
F[s_,n_]:=Module[{u,m},u=Sort[Eigenvalues[1.0*Laplace[s,n]]];
  m=Length[u]; Plot[u[[Floor[x*m]]],{x,0,1}]];
TopDim=1; Show[Table[F[CycleGraph[4],k],{k,3}],PlotRange->{0,4}]
TopDim=2; Show[Table[F[CompleteGraph[3],k],{k,3}],PlotRange->{0,10}]
\end{lstlisting}
\end{tiny}

And here is the recursive computation of the matrix $A$ which gives the 
clique data $A \vec{v}$ of the Barycentric refinement of $G_1$ 
if $\vec{v}$ is the clique vector of $G$. The computation produces the 
clique data of the boundary of $(K_n)_1$ which allows to compute the 
number of interior $k$-simplices producing the off diagonal matrix 
elements not in the top row. 

\begin{tiny}
\lstset{language=Mathematica} \lstset{frameround=fttt}
\begin{lstlisting}[frame=single]
BarycentricOperator[m_]:=Module[{},
b[A_]:=Module[{n=Length[A],c},c=A.Table[Binomial[n+1,k],{k,n}];
    Delete[Prepend[c,1],n+1]];
T[A_]:=Append[Transpose[Append[Transpose[A],b[A]]],
    Append[Table[0,{Length[A]}],(Length[A]+1)!]];
Last[NestList[T,{{1}},m]]]; BarycentricOperator[7] 
\end{lstlisting}
\end{tiny}

$$ A = \left[
                 \begin{array}{cccccccc}
                  1 & 1 & 1 & 1 & 1 & 1 & 1 & 1 \\
                  0 & 2 & 6 & 14 & 30 & 62 & 126 & 254 \\
                  0 & 0 & 6 & 36 & 150 & 540 & 1806 & 5796 \\
                  0 & 0 & 0 & 24 & 240 & 1560 & 8400 & 40824 \\
                  0 & 0 & 0 & 0 & 120 & 1800 & 16800 & 126000 \\
                  0 & 0 & 0 & 0 & 0 & 720 & 15120 & 191520 \\
                  0 & 0 & 0 & 0 & 0 & 0 & 5040 & 141120 \\
                  0 & 0 & 0 & 0 & 0 & 0 & 0 & 40320 \\
                 \end{array}
                 \right]  \; . $$

The procedure produces finite dimensional versions of this matrix which 
matter when looking for interesting quantities on $\G_d$. For $d=2$ for example, we have 
$\left[ \begin{array}{ccc} 1 & 1 & 1 \\ 0 & 2 & 6 \\ 0 & 0 & 6 \\ \end{array} \right]$
whose transpose has the eigenvectors $[1,-1,1]^T$ (Euler characteristic), the 
eigenvectors $[0,-2,3]^T$, a well known invariant for $2$-graphs which is proportional
to the boundary for $2$-graphs with boundary as well as $[0,0,1]^T$ which is
area. For $d=3$, where $A=\left[ \begin{array}{cccc} 1 & 1 & 1 & 1 \\
                   0 & 2 & 6 & 14 \\ 0 & 0 & 6 & 36 \\ 0 & 0 & 0 & 24 \\ \end{array} \right]$,
the eigenvectors besides the Euler characteristic vector $[1,-1,1,-1]^T$ to the eigenvalue $1$
and volume $[0,0,0,1]^T$ to the eigenvalue $24$, there is $[0,22,-33,40]^T$ with 
eigenvalue $2$ and $[0,0,0-1,2]^T$ with eigenvalue $6$. The later gives an invariant which is
zero for $3$-graphs as in general $[0,\dots 0,-2,d+1]$ is an invariant for $d$-graphs.
We see experimentally that {\bf for any $d$-graph,the eigenfunction of $A^T$ to each eigenvalue $(2k)!$ 
is perpendicular to the clique vector $\vec{v}$}. For example, there is a discrete $P^2 \times S^2$ with 
clique vector $\vec{v}=[1908, 26520, 87020, 104010, 41604]$ which is
perpendicular to $\vec{v}_4=[0, -22, 33, -40, 45]$. The invariant does not change under
{\bf edge refinement modifications} which are homotopies preserving
$d$-graphs. The invariants $X_{2k}(G) = \vec{v}(G) \cdot \vec{v}_{2k}$ remain zero under such homotopies.
They are currently under investigation. It looks promising that integral theoretical methods like
\cite{indexformula,colorcurvature} for generalized curvatures allow to prove the invariants to be zero.
It would be useful also to know whether any two homeomorphic $d$-graphs have a common refinement when
using Barycentric or edge refinements. This looks more accessible than related questions for 
triangulations as $d$-graphs can be dealt with recursively and refinements of unit spheres 
carry to refinements of the entire graph. 
 
\bibliographystyle{plain}

\begin{thebibliography}{10}

\bibitem{Brouwer}
W.H.~Haemers A.E.~Brouwer.
\newblock {\em Spectra of graphs}.
\newblock Springer, 2010.

\bibitem{BR}
R.~Balakrishnan and K.~Ranganathan.
\newblock {\em A textbook of Graph Theory}.
\newblock Springer, 2012.

\bibitem{BandeltVandeVel}
H.-J. Bandelt and M.~van~de Vel.
\newblock Embedding topological median algebras in products of dendrons.
\newblock {\em Proc. London Math. Soc. (3)}, 58(3):439--453, 1989.

\bibitem{Bollobas1}
B.Bollob{\'a}s.
\newblock {\em Modern Graph Theory}.
\newblock Graduate Texts in Mathematics. Springer, New York, 1998.

\bibitem{hexacarpet}
M.~Begue, D.J. Kelleher, A.~Nelson, H.~Panzo, R.~Pellico, and A.~Teplyaev.
\newblock Random walks on barycentric subdivisions and the {S}trichartz
  hexacarpet.
\newblock {\em Exp. Math.}, 21(4):402--417, 2012.

\bibitem{BergerPanorama}
M.~Berger.
\newblock {\em A Panoramic View of Riemannian Geometry}.
\newblock Springer Verlag, Berlin, 2003.

\bibitem{Chavel}
I.~Chavel.
\newblock {\em Eigenvalues in Riemannian Geometry}.
\newblock Pure and applied mathematics. Academic Press Inc., Orlando, 1984.

\bibitem{CYY}
B.~Chen, S-T. Yau, and Y-N. Yeh.
\newblock Graph homotopy and {G}raham homotopy.
\newblock {\em Discrete Math.}, 241(1-3):153--170, 2001.
\newblock Selected papers in honor of Helge Tverberg.

\bibitem{Chung97}
F.~Chung.
\newblock {\em Spectral graph theory}, volume~92 of {\em CBMS Regional Conf.
  Series}.
\newblock AMS, 1997.

\bibitem{CFS}
I.P. Cornfeld, S.V.Fomin, and Ya.G.Sinai.
\newblock {\em Ergodic Theory}, volume 115 of {\em {Grundlehren} der
  mathematischen {Wissenschaften} in {Einzeldarstellungen}}.
\newblock Springer Verlag, 1982.

\bibitem{Cycon}
H.L. Cycon, R.G.Froese, W.Kirsch, and B.Simon.
\newblock {\em {Schr\"odinger} Operators---with Application to Quantum
  Mechanics and Global Geometry}.
\newblock Springer-Verlag, 1987.

\bibitem{DeMelo}
W.~de~Melo and S.~van Strien.
\newblock {\em One dimensional dynamics}, volume~25 of {\em Series of modern
  surveys in mathematics}.
\newblock Springer Verlag, 1993.

\bibitem{VerdiereGraphSpectra}
Y.Colin de~Verdi{\`e}re.
\newblock Spectres de graphes.
\newblock 1998.

\bibitem{DGS}
M.~Denker, C.~Grillenberger, and K.~Sigmund.
\newblock {\em Ergodic Theory on Compact Spaces}.
\newblock Lecture Notes in Mathematics 527. Springer, 1976.

\bibitem{DiaconisMiclo}
P.~Diaconis and L.~Miclo.
\newblock On barycentric subdivision.
\newblock {\em Combin. Probab. Comput.}, 20(2):213--237, 2011.

\bibitem{Dieudonne1989}
J.~Dieudonne.
\newblock {\em A History of Algebraic and Differential Topology, 1900-1960}.
\newblock Birkh\"auser, 1989.

\bibitem{Evako1994}
A.V. Evako.
\newblock Dimension on discrete spaces.
\newblock {\em Internat. J. Theoret. Phys.}, 33(7):1553--1568, 1994.

\bibitem{Evako2013}
A.V. Evako.
\newblock The {Jordan-Brouwer} theorem for the digital normal n-space space
  {$Z^n$}.
\newblock http://arxiv.org/abs/1302.5342, 2013.

\bibitem{Feigenbaum1978}
M.J. Feigenbaum.
\newblock Quantitative universality for a class of nonlinear transformations.
\newblock {\em J. Statist. Phys.}, 19(1):25--52, 1978.

\bibitem{Friedman}
N.A. Friedman.
\newblock {\em Introduction to Ergodic Theory}.
\newblock {Van Nostrand-Reinhold, Princeton, New York}, 1970.

\bibitem{BBM}
G.A.Baker, D.Bessis, and P.Moussa.
\newblock A family of almost periodic {Schr\"o}dinger operators.
\newblock {\em Physica A}, 124:61--78, 1984.

\bibitem{handbookgraph}
J.~Gross and J.~Yellen, editors.
\newblock {\em Handbook of graph theory}.
\newblock Discrete Mathematics and its Applications (Boca Raton). CRC Press,
  Boca Raton, FL, 2004.

\bibitem{TuckerGross}
J.L. Gross and T.W. Tucker.
\newblock {\em Topological Graph Theory}.
\newblock John Wiley and Sons, 1987.

\bibitem{Hamelink68}
R.C. Hamelink.
\newblock A partial characterization of clique graphs.
\newblock {\em Journal of Combinatorial Theory}, 5:192--197, 1968.

\bibitem{Hatcher}
A.~Hatcher.
\newblock {\em Algebraic Topology}.
\newblock Cambridge University Press, 2002.

\bibitem{Hough}
B.~Hough.
\newblock Tessellation of a triangle by repeated barycentric subdivision.
\newblock {\em Electron. Commun. Probab.}, 14:270--277, 2009.

\bibitem{Lanford84}
O.E.~Lanford III.
\newblock A shorter proof of the existence of the {F}eigenbaum fixed point.
\newblock {\em Commun. Math. Phys}, 96:521--538, 1984.

\bibitem{I94}
A.~Ivashchenko.
\newblock Contractible transformations do not change the homology groups of
  graphs.
\newblock {\em Discrete Math.}, 126(1-3):159--170, 1994.

\bibitem{I94a}
A.V. Ivashchenko.
\newblock Graphs of spheres and tori.
\newblock {\em Discrete Math.}, 128(1-3):247--255, 1994.

\bibitem{HHM}
J.~Hirst J.~Harris and M.~Mossinghoff.
\newblock {\em Combinatorics and Graph Theory}.
\newblock Springer, 2008.

\bibitem{KellendonkLenzSavinien}
J.~Savinien J.~Kellendonk, D.~Lenz.
\newblock {\em Mathematics of Aperiodic Order}, volume 309 of {\em Progress in
  Mathematics}.
\newblock Birkhaeuser, 2015.

\bibitem{KlainRota}
D.A. Klain and G-C. Rota.
\newblock {\em Introductioni to geometric probability}.
\newblock Lezioni Lincee. Accademia nazionale dei lincei, 1997.

\bibitem{Kni93b}
O.~Knill.
\newblock Factorisation of random {Jacobi} operators and {B\"acklund}
  transformations.
\newblock {\em Communications in Mathematical Physics}, 151:589--605, 1993.

\bibitem{Kni93a}
O.~Knill.
\newblock Isospectral deformations of random {Jacobi} operators.
\newblock {\em Communications in Mathematical Physics}, 151:403--426, 1993.

\bibitem{Kni95}
O.~Knill.
\newblock Renormalization of of random {Jacobi} operators.
\newblock {\em Communications in Mathematical Physics}, 164:195--215, 1995.

\bibitem{cherngaussbonnet}
O.~Knill.
\newblock A graph theoretical {Gauss-Bonnet-Chern} theorem.
\newblock {\\}http://arxiv.org/abs/1111.5395, 2011.

\bibitem{indexformula}
O.~Knill.
\newblock An index formula for simple graphs \hfill.
\newblock {\\}http://arxiv.org/abs/1205.0306, 2012.

\bibitem{KnillILAS}
O.~Knill.
\newblock The {D}irac operator of a graph.
\newblock {{\\}http://arxiv.org/abs/1306.2166}, 2013.

\bibitem{colorcurvature}
O.~Knill.
\newblock Curvature from graph colorings.
\newblock {{\\}http://arxiv.org/abs/1410.1217}, 2014.

\bibitem{KnillBarycentric}
O.~Knill.
\newblock The graph spectrum of barycentric refinements.
\newblock {{\\}http://arxiv.org/abs/1508.02027}, 2015.

\bibitem{KnillEulerian}
O.~Knill.
\newblock Graphs with {E}ulerian unit spheres.
\newblock http://arxiv.org/abs/1501.03116, 2015.

\bibitem{KnillJordan}
O.~Knill.
\newblock The {J}ordan-{B}rouwer theorem for graphs.
\newblock {{\\}http://arxiv.org/abs/1506.06440}, 2015.

\bibitem{KnillProduct}
O.~Knill.
\newblock The {K}uenneth formula for graphs.
\newblock http://arxiv.org/abs/1505.07518, 2015.

\bibitem{KnillSard}
O.~Knill.
\newblock A {S}ard theorem for graph theory.
\newblock {{\\}http://arxiv.org/abs/1508.05657}, 2015.

\bibitem{Last1995}
Y.~Last.
\newblock Personal communication.
\newblock 1995.

\bibitem{LewisShisha}
J.T. Lewis and O.~Shisha.
\newblock Lp convergence of monotone functions and their uniform convergence.
\newblock {\em Journal of Approximation Theory}, 14:281--284, 1975.

\bibitem{McKeanSinger}
H.P. McKean and I.M. Singer.
\newblock Curvature and the eigenvalues of the {L}aplacian.
\newblock {\em J. Differential Geometry}, 1(1):43--69, 1967.

\bibitem{McMullenSchulte}
P.~McMullen and E.~Schulte.
\newblock {\em Abstract Regular Polytopes}.
\newblock Encyclopedia of Mathematics and its applications. Cambridge
  University Press, 2002.

\bibitem{BGH}
M.F.Barnsley, J.S.Geronimo, and A.N. Harrington.
\newblock Almost periodic jacobi matrices associated with julia sets for
  polynomials.
\newblock {\em Commun. Math. Phys.}, 99:303--317, 1985.

\bibitem{Pastur}
L.~Pastur and A.Figotin.
\newblock {\em Spectra of Random and Almost-Periodic Operators}, volume 297.
\newblock Springer-Verlag, Berlin--New York, {Grundlehren} der mathematischen
  {Wissenschaften} edition, 1992.

\bibitem{Post}
O.~Post.
\newblock {\em Spectral Analysis on Graph like Spaces}, volume 2039 of {\em
  Lecture notes in Mathematics}.
\newblock 2012.

\bibitem{Mieghem}
P.VanMieghem.
\newblock {\em Graph Spectra for complex networks}.
\newblock Cambridge University Press, 2011.

\bibitem{Carmona}
J.Lacroix R.~Carmona.
\newblock {\em Spectral Theory of Random {S}chr\"odinger Operators}.
\newblock Birkh\"auser, 1990.

\bibitem{Rosenberg}
S.~Rosenberg.
\newblock {\em The Laplacian on a Riemannian Manifold}, volume~31 of {\em
  London Mathematical Society, Student Texts}.
\newblock Cambridge University Press, 1997.

\bibitem{Rotman}
J.J. Rotman.
\newblock {\em An introduction to Algebraic Topology}.
\newblock Graduate Texts in Mathematics. Springer.

\bibitem{Santalo}
L.A. Santalo.
\newblock {\em Introduction to integral geometry}.
\newblock Hermann and Editeurs, Paris, 1953.

\bibitem{Senechal}
M.~Senechal.
\newblock {\em Quasicrystals and geometry}.
\newblock Cambridge University Press, 1995.

\bibitem{Simon82}
B.~Simon.
\newblock Almost periodic schroedinger operators: A review.
\newblock {\em Advances in Mathematics}, 3:463--390, 1982.

\bibitem{SimonTrace}
B.~Simon.
\newblock {\em Trace Ideals and Their Applications}.
\newblock AMS, 2. edition, 2010.

\bibitem{Snyder}
D.F. Snyder.
\newblock Combinatorics of barycentric subdivision and characters of simplicial
  two-complexes.
\newblock {\em Amer. Math. Monthly}, 113(9):822--826, 2006.

\bibitem{BLS}
J.~Leydold T.~Bijikoglu and P.~Stadler.
\newblock {\em Laplacian Eigenvectors of Graphs}, volume 1915 of {\em Lecture
  Notes in Mathematics}.
\newblock Springer, 2007.

\bibitem{TaoDyadic}
T.~Tao.
\newblock Dyadic models.
\newblock https://terrytao.wordpress.com/2007/07/27/dyadic-models, 2007.

\bibitem{Whitehead}
J.H.C. Whitehead.
\newblock Simplicial spaces, nuclei and m-groups.
\newblock {\em Proc. London Math. Soc.}, 45(1):243--327, 1939.

\end{thebibliography}

\end{document}